\documentclass[final]{siamltex}

\usepackage{amssymb}
\usepackage{amsmath}
\usepackage{amsfonts}
\usepackage{algorithm}
\usepackage{graphicx}
\usepackage{color}
\usepackage{epsfig}

\numberwithin{algorithm}{section}
\newtheorem{remark}{Remark}
\newtheorem{example}{Example}

\newcommand{\bw}{\mathbf w}
\newcommand{\bx}{\mathbf x}

\begin{document}

\title{A modified Newton iteration for finding nonnegative $Z$-eigenpairs  of a nonnegative tensor}
\author{ Chun-Hua Guo\thanks{%
Department of Mathematics and Statistics, University of Regina, Regina, SK
S4S 0A2, Canada (\texttt{Chun-Hua.Guo@uregina.ca}). This author was
supported in part by an NSERC Discovery Grant.} \and Wen-Wei Lin\thanks{%
Department of Applied Mathematics, National Chiao Tung University, Hsinchu
300, Taiwan (\texttt{wwlin@math.nctu.edu.tw}). This author was supported in
part by the Ministry of Science and Technology, the National Center for
Theoretical Sciences, and ST Yau Center at Chiao-Da in Taiwan.} \and %
Ching-Sung Liu\thanks{%
Department of Applied Mathematics, National University of Kaohsiung,
Kaohsiung 811, Taiwan (\texttt{chingsungliu@nuk.edu.tw}). This author was
supported in part by the Ministry of Science and Technology.} }
\maketitle

\begin{abstract}
We propose a modified Newton iteration for finding  some nonnegative $Z$-eigenpairs of a nonnegative tensor.
When the tensor is irreducible, all nonnegative eigenpairs are known to be positive. We prove local quadratic
convergence of the new iteration to any positive eigenpair of a nonnegative tensor, under the usual assumption guaranteeing the local quadratic
convergence of the original Newton iteration. A big advantage of the modified Newton iteration is that it seems capable of finding a 
nonnegative eigenpair starting with any positive unit vector. 
Special attention is paid to transition probability tensors.
\end{abstract}



\section{Introduction}
\label{sec1}

A real-valued $m$th-order $n$-dimensional tensor $\mathcal{A}$ consists of $%
n^{m}$ entries in $\mathbb{R}$, and has the form
\begin{equation*}
\mathcal{A}=(A_{i_{1}i_{2}\ldots i_{m}})\text{, }\quad A_{i_{1}i_{2}\ldots
i_{m}}\in \mathbb{R}\text{, }\quad 1\leq i_{1},i_{2},\ldots ,i_{m}\leq n.
\end{equation*}%
The set of all such tensors is denoted by $\mathbb{R}^{[m,n]}$. We use $\mathbb{R}_{+}^{[m,n]}$ to denote 
all nonnegative tensors $\mathcal{A}\in \mathbb{R}^{[m,n]}$, for which $A_{i_{1}i_{2}\ldots
i_{m}}\geq 0$ for all $i_{1}, i_{2}, \ldots, 
i_{m}$. 
Various applications of tensors, nonnegative tensors in
particular, can be found in \cite{KB09}.

For a column vector $\mathbf{x}=[x_{1},x_{2},\ldots
,x_{n}]^{T}\in \mathbb{R}^{n},$ we define a column vector in $\mathbb{R}^n$ by 
\begin{equation}
\mathcal{A}\mathbf{x}^{m-1}:=\left( \sum_{i_{2},\ldots
,i_{m}=1}^{n}A_{ii_{2}\ldots i_{m}}x_{i_{2}}\ldots x_{i_{m}}\right) _{1\leq
i\leq n}.  \label{eq: Ax}
\end{equation}

We are interested in eigenvalue problems for nonnegative tensors. 

\begin{definition}[\protect\cite{Q05,CZ13}]
Let $\mathcal{A}\in \mathbb{R}^{[m,n]}$. 
 We say that $(\mathbf{x}, \lambda )\in \left( \mathbb{R}^{n}\backslash \{0\}\right)  \times \mathbb{R}$ is an $H$-eigenpair
(eigenvector-eigenvalue) of $\mathcal{A}$ if%
\begin{equation}
\mathcal{A}\mathbf{x}^{m-1}=\lambda \mathbf{x}^{[m-1]},  \label{eq: NEPH}
\end{equation}
 where $\mathbf{x}^{[m-1]}=[x_{1}^{m-1},x_{2}^{m-1},\ldots ,x_{n}^{m-1}]^{T},$
and is a $Z$-eigenpair of $\mathcal{A}$ if
\begin{equation}
\mathcal{A}\mathbf{x}^{m-1}=\lambda \mathbf{x}, \quad \|\mathbf{x}\|=1.  \label{eq: NEP}
\end{equation}
\end{definition}

If $\mathbf{x}$ is an $H$-eigenvector, then $c\mathbf{x}$ is also an $H$-eigenvector 
for any $c\ne 0$. The same is not true in general for $Z$-eigenvectors. That is why we require 
$\|\mathbf{x}\|=1$ in \eqref{eq: NEP} with $\|\cdot \|$ being any vector norm. If the $2$-norm is used, then a $Z$-eigenpair is called a $Z_2$-eigenpair; 
If the $1$-norm is used, then a $Z$-eigenpair is called a $Z_1$-eigenpair. 
As noted in \cite{CZ13}, for $\mathbf{x}$ with $\|\mathbf{x}\|_1=1$, 
$(\mathbf{x}, \lambda)$ is a $Z_1$-eigenpair if and only if 
$\left ( \frac{\mathbf{x}}{\|\mathbf{x}\|_2},  \frac{\lambda}{\|\mathbf{x}\|_2^{m-2}} \right )$
 is a $Z_2$-eigenpair. In this paper, we are interested in $Z_1$-eigenpairs since special attention will be paid to transition probability tensors. A $Z_1$-eigenpair will be referred to as a $Z$-eigenpair or simply an eigenpair.

A weakly irreducible nonnegative tensor has a unique positive $H$-eigenvector $\mathbf{x}_{*}$ (up to a positive scalar multiple) 
and the corresponding eigenvalue $\lambda_{*}$ is positive \cite{FGH}. The positive $H$-eigenpair $(\mathbf{x}_{*}, \lambda_{*})$ may be found by the NQZ algorithm 
\cite{NQZ09}, whose (linear) convergence is guaranteed for the smaller class of weakly primitive tensors \cite{HHQ11}. In \cite{LGL15,LGL17}, we  present a modified Newton iteration, called the Newton--Noda iteration, for finding the unique positive $H$-eigenpair. The method requires the selection of a positive parameter $\theta_k$ in the $k$th iteration, and naturally keeps the positivity in the approximate eigenpairs. 
For $m=3$, a practical procedure for choosing $\theta_k$ is given in \cite{LGL15}, which guarantees the global convergence of the method. 
For a general $m$, a different practical procedure for choosing $\theta_k$ is given in \cite{LGL17}, and the global convergence of the method is almost certain. Both procedures will give $\theta_k=1$ near convergence and local quadratic convergence is achieved. 
The benefit of using $\theta_k=1$ right from the beginning is also mentioned in \cite{LGL17}, but the global convergence of the method becomes less certain in this case, although no examples showing divergence have been found. 

The $Z$-eigenvalue problem is much more difficult. When the tensor is irreducible, all nonnegative $Z$-eigenpairs are positive
but there may be many such pairs \cite{CPZ13}, so global convergence of any iterative method to a fixed positive eigenpair becomes impossible in general.  A main algorithm for the $Z$-eigenvalue problem has been  the shifted symmetric higher-order power method (SS-HOPM) in \cite{KM11}. 

After some preliminaries in Section \ref{sec2}, we propose in Section \ref{sec3} a modified Newton iteration for finding a nonnegative $Z$-eigenpair of a nonnegative tensor, 
in the spirit of \cite{LGL17} for the $H$-eigenvalue problem. 
If we compare our algorithm here to that in 
\cite{LGL17} (although they are for two different problems), we no longer try to select parameters $\theta_k$ to guarantee the 
monotonic convergence of the sequence approximating a nonnegative $Z$-eigenvalue and we effectively use $\theta_k=1$ all the time here. 
  When the tensor has more than one nonnegative eigenpairs, we expect to find some of them by using different initial vectors in our algorithm. 
Unlike the algorithm in \cite{LGL17} for the $H$-eigenvalue problem, our algorithm here does not naturally preserve 
nonnegativity in approximate $Z$-eigenpairs. Instead, the nonnegativity is preserved through 
some intervention when needed. 
In Section \ref{sec4} we prove local quadratic
convergence of the new iteration to any positive eigenpair of a nonnegative tensor, under the usual assumption guaranteeing the local quadratic
convergence of the original Newton iteration. 
In Section \ref{sec5} we pay special attention to transition probability tensors and explain why in this special case nonnegativity can often be preserved 
without the intervention. 
The usefulness of our new algorithm will be illustrated through some  numerical results in Section \ref{sec6}. 
Some concluding remarks are given in Section \ref{sec7}.

\section{Preliminaries}
\label{sec2}

Nonnegative and positive matrices or vectors are defined entrywise. For example, 
$A=\left[ A_{ij}\right]$ is nonnegative, written $A\ge 0$,  if $A_{ij}\geq 0$ for all $i$ and $j$. 
A $Z$-matrix is a real square matrix whose off-diagonal elements are all nonpositive. 
A $Z$-matrix $A$ is called a nonsingular $M$-matrix
if $A=sI-B$ with $B\geq 0$ and $s>\rho (B)$, where $\rho (\cdot )$ is the spectral radius. 
A $Z$-matrix $A$ is a nonsingular $M$-matrix if and only if  $A^{-1}\ge 0$ 
(see \cite{BPl94} for example).

In this paper all vectors are $n$-vectors and all matrices are $n\times n$, unless
specified otherwise.
We use $v_i$ or $(\mathbf{v})_i$  to represent the $i$th element of a vector $\mathbf{v}$.
For a pair of positive vectors $\mathbf{v}$ and $\mathbf{w}$, we define
\begin{equation*}
\max \left( \frac{\mathbf{w}}{\mathbf{v}}\right) =\underset{i}{\max }
\left (\frac{w_i}{v_i}\right ) ,\text{ \ }\min \left(
\frac{\mathbf{w}}{\mathbf{v}}\right) =\underset{i}{\min }\left (\frac{w_i}{v_i}\right ). 
\end{equation*}%

 We will sometimes assume a tensor in  $\mathbb{R}_{+}^{[m,n]}$ is  irreducible or weakly irreducible.

\begin{definition}[\protect\cite{CPZ08,NQZ09}]
A tensor $\mathcal{A}\in \mathbb{R}^{[m,n]}$   is called reducible if
there exists a nonempty proper index subset $S\subset \left\{ 1,2,\ldots
,n\right\} $ such that%
\begin{equation*}
A_{i_{1}i_{2}\ldots i_{m}}=0,\text{ }\forall \ i_{1}\in S,\text{ }\forall \
i_{2},\ldots ,i_{m}\notin S.
\end{equation*}%
If $\mathcal{A}$ is not reducible, then we call $\mathcal{A}$ irreducible.
\end{definition}

The notion of weakly irreducible nonnegative tensors is
introduced in \cite{FGH}. The following equivalent definition is given in \cite{YY11}. 

\begin{definition}
\label{weakly2}
A tensor $\mathcal{A}\in \mathbb{R}_+^{[m,n]}$   is called weakly irreducible if for every nonempty proper index subset $S\subset \left\{ 1,2,\ldots,n\right\}$ 
there exist $i_1\in S$ and $i_2, \ldots, i_m$ with at least one $i_q\notin S$, $q=2, \ldots, m$, such that 
$A_{i_{1}i_{2}\ldots i_{m}}\ne 0$. 
\end{definition}

Note that all irreducible tensors in $\mathbb{R}_+^{[m,n]}$ are weakly irreducible.

The following result is given in Theorems 2.5 and 2.6 of \cite{CPZ13}.

\begin{theorem}
If $\mathcal{A} \in \mathbb{R}_{+}^{[m,n]}$,  then $\mathcal{A}$ has a nonnegative 
 $Z$-eigenpair $(\mathbf{x}, \lambda)$. 
If  $\mathcal{A}$ is irreducible, then every nonnegative 
 $Z$-eigenpair is  positive.
\end{theorem}

A tensor $\mathcal{A} \in \mathbb{R}_{+}^{[m,n]}$ is said to be semisymmetric \cite{NQ15}
if $A_{i j_2 \ldots j_m}=A_{i i_2 \ldots i_m}$, $1\le i\le n$, $j_2\ldots
j_m$ is any permutation of $i_2\ldots i_m$, $1\le i_2, \ldots, i_m\le n$.
For any $\mathcal{A} \in \mathbb{R}_{+}^{[m,n]}$, we can get a semisymmetric tensor ${\mathcal{A}}_s=(A^{(s)}_{i i_{2}\ldots i_{m}} ) \in \mathbb{R}_{+}^{[m,n]}$ such that $%
\mathcal{A}\mathbf{x}^{m-1}=\mathcal{A}_s\mathbf{x}^{m-1}$, by an averaging
procedure. Specifically, 
for any $1\le i, i_2, \ldots, i_m\le n$, let $j_2^{(1)}\ldots j_m^{(1)}, \ldots, j_2^{(q)}\ldots j_m^{(q)}$ 
be all different permutations of $i_2\ldots i_m$ (we have $q\le (m-1)!$ since some of the $i_k$'s may be the same). 
Then we define $A^{(s)}_{i i_{2}\ldots i_{m}}=\frac{1}{q}\sum_{k=1}^q A_{i j_2^{(k)}\ldots j_m^{(k)}}$. 
The total computational work for obtaining ${\mathcal{A}}_s$ is about $n^m$ flops.

We are going to find an eigenpair $(\mathbf{x}_{\ast}, \lambda_{\ast})$ with
$\mathbf{x}_{\ast}\ge 0$ and $\|\mathbf{x}_{\ast}\|_1=\mathbf{e}^T\mathbf{x}%
_{\ast}=1$, where $\mathbf{e}=[1,\ldots ,1]^{T}.$

We define two vector valued functions $\mathbf{r}:$ $\mathbb{R}_{+}^{n+1}%
\mathbb{\rightarrow R}^{n}$ and \newline
$\mathbf{f}:$ $\mathbb{R}_{+}^{n+1}\mathbb{\rightarrow R}^{n+1}$ as follows:
\begin{equation}
\mathbf{r}(\mathbf{x,}\lambda )=\lambda \mathbf{x}-\mathcal{A}\mathbf{x}%
^{m-1}, \quad \mathbf{f}(\mathbf{x},\lambda )=\left[
\begin{array}{c}
\mathbf{r}(\mathbf{x,}\lambda ) \\
\mathbf{e}^{T}\mathbf{x}-1%
\end{array}%
\right].  \label{eq:Fx}
\end{equation}

Then the Jacobian of $\mathbf{f}(x,\lambda )$ is given by
\begin{equation}
\mathbf{Jf}(\mathbf{x},\lambda )=\left[
\begin{array}{cc}
\lambda I-(m-1)T(\mathbf{x}) & \mathbf{x} \\
\mathbf{e}^{T} & 0%
\end{array}%
\right] ,  \label{eq: graF}
\end{equation}%
where the entries of $T(\mathbf{x})$ are
\begin{equation*}  
T(\mathbf{x})_{ij}=\frac{1}{m-1}\frac{\partial}{\partial x_j}\left( \mathcal{%
A}\mathbf{x}^{m-1}\right)_i.
\end{equation*}
When the tensor is semisymmetric, we have 
by the proof of \cite[Lemma 3.3]{NQ15} that
\begin{equation}  \label{mT}
T(\mathbf{x})_{ij}= \sum_{i_3, \ldots, i_m=1}^n A_{i j i_{3}\ldots
i_{m}}x_{i_{3}}\ldots x_{i_{m}}, 
\end{equation}
from which we obtain 
\begin{equation}  \label{Txx}
T(\mathbf{x}){\mathbf{x}}=\mathcal{A}{\mathbf{x}}^{m-1}.
\end{equation}
Note that \eqref{Txx} holds even when $\mathcal{A}$ is not semisymmetric since 
$\mathcal{A}\mathbf{x}^{m-1}=\mathcal{A}_s\mathbf{x}^{m-1}$. 

The following result has been proved in \cite{LGL17}.
\begin{lemma}
\label{irreducible} Let $\mathcal{A} \in \mathbb{R}_{+}^{[m,n]}$ be  weakly irreducible and $\mathbf{x}$ be a positive vector. Then the
nonnegative matrix $T(\mathbf{x})$ is irreducible.
\end{lemma}

In this paper we will pay special attention to  transition probability tensors. 
\begin{definition}[\protect\cite{CZ13, LN14}]
A tensor $\mathcal{A}  \in \mathbb{R}_{+}^{[m,n]}$ is called a transition probability
tensor if $\mathcal{A}=(A_{i_{1}i_{2}\ldots i_{m}})$ satisfies%
\begin{equation*}
\sum_{i=1}^{n}A_{ii_{2}\ldots i_{m}}=1,\text{ } 1\le i_2, \ldots, i_m\le n. 
\end{equation*}
\end{definition}

Here is a main theoretical result about the $Z$-eigenvalue problem for transition probability tensors.  
\begin{theorem}[\protect\cite{CZ13, LN14}]
Let  $\mathcal{A} \in \mathbb{R}_{+}^{[m,n]}$  be a  transition probability tensor. 
 Then $1$ is the unique $Z$-eigenvalue of  $\mathcal{A}$ with a corresponding nonnegative eigenvector $\mathbf{x}$.  If $\mathcal{A}$   is irreducible, then every nonnegative eigenvector $\mathbf{x}$  must be positive. 
\end{theorem}

The following result will be needed shortly. 

\begin{lemma}\label{ss} 
Let $\mathcal{A} \in \mathbb{R}_{+}^{[m,n]}$  be a  transition probability tensor and 
$\mathcal{A}_s=(A^{(s)}_{i_1 i_2 \ldots i_m})$ be the semisymmetric tensor obtained from $\mathcal{A}$ by an averaging procedure. Then $\mathcal{A}_s$  is also a  transition probability tensor. 
\end{lemma}
\begin{proof}
For any $1\le i, i_2, \ldots, i_m\le n$, let $j_2^{(1)}\ldots j_m^{(1)}, \ldots, j_2^{(q)}\ldots j_m^{(q)}$ 
be all different permutations of $i_2\ldots i_m$. 
Then 
$$
\sum_{i=1}^{n}A^{(s)}_{ii_{2}\ldots i_{m}}
=\sum_{i=1}^{n}\left (\frac{1}{q}\sum_{k=1}^q A_{ij_2^{(k)}\ldots j_m^{(k)}}\right )
=
\frac{1}{q}\sum_{k=1}^q  \left (\sum_{i=1}^{n} A_{ij_2^{(k)}\ldots j_m^{(k)}}\right )
=\frac{1}{q}\sum_{k=1}^q 1=1, 
$$
as required. \end{proof}

The following result is given in \cite[Lemma 5.2]{CZ13}, but the proof there is incomplete. 

\begin{lemma}
\label{leftev}
Let $\mathcal{A} \in \mathbb{R}_{+}^{[m,n]}$  be a transition probability tensor and $\mathbf{x}$ be a positive vector with $%
\left\Vert \mathbf{x}\right\Vert _{1}=1$. Then 
$\mathbf{e}^{T}T(\mathbf{x})=  \mathbf{e}^{T}$, i.e., 
$T(\mathbf{x})$ is a (column) stochastic matrix. 
\end{lemma}

\begin{proof}
The proof in \cite{CZ13} starts with the equality in \eqref{mT}, which does not hold in general 
when $\mathcal{A}$ is not semisymmetric. 
Let $\mathcal{A}_s=(A^{(s)}_{i_1 i_2 \ldots i_m})$ be the semisymmetric tensor obtained from $\mathcal{A}$ by an averaging procedure. 
Then 
\begin{equation*}  
T(\mathbf{x})_{ij}=\frac{1}{m-1}\frac{\partial}{\partial x_j}\left( \mathcal{%
A}\mathbf{x}^{m-1}\right)_i= \frac{1}{m-1}\frac{\partial}{\partial x_j} \left (\mathcal{%
A}_s\mathbf{x}^{m-1}\right)_i=
\sum_{i_3, \ldots, i_m=1}^n A^{(s)}_{i j i_{3}\ldots
i_{m}}x_{i_{3}}\ldots x_{i_{m}}.
\end{equation*}
By Lemma \ref{ss},  $\mathcal{A}_s$ is still a 
transition probability tensor. 
 Now, a direct computation shows that $\left( \mathbf{e}^{T}T(\mathbf{x})\right) _{j}=1$ for each $j$, as in \cite{CZ13}. 
\end{proof}

We also have the following inclusion result for the $Z$-eigenvalue $1$ of a  transition probability tensor.

\begin{lemma}
\label{minmax}
Let $\mathcal{A} \in \mathbb{R}_{+}^{[m,n]}$  be a transition probability tensor. For any positive vector $\mathbf{v}$ with $\left\Vert
\mathbf{v}\right\Vert _{1}=1$, we have
\begin{equation*}
\min \left( \frac{\mathcal{A}\mathbf{v}^{m-1}}{\mathbf{v}}\right) \leq1 \leq \max \left( \frac{\mathcal{A}\mathbf{v}^{m-1}}{\mathbf{v%
}}\right) .
\end{equation*}
\end{lemma}

\begin{proof}
By the Perron--Frobenius theorem for nonnegative matrices \cite{BPl94,V00}, we get%
\begin{equation*}
\min  \left( \frac{\mathcal{A}\mathbf{v}^{m-1}}{\mathbf{v}}\right) 
 =\min \left( \frac{T(\mathbf{v})\mathbf{v}}{\mathbf{v}}\right) \leq \rho (T(\mathbf{v}))\leq \max
\left( \frac{T(\mathbf{v})\mathbf{v}}{\mathbf{v}}\right)=
\max \left( \frac{\mathcal{A}\mathbf{v}^{m-1}}{\mathbf{v}}\right). 
\end{equation*}%
Since $T(\mathbf{v})$ is a stochastic matrix by Lemma \ref{leftev}, we have
$\rho (T(\mathbf{v}))=1$. 
\end{proof}

\section{A modified Newton iteration}\label{sec3}

In this section we present a modified Newton iteration for finding a nonnegative eigenpair of a nonnegative tensor $\mathcal{A}$. 
In the derivation, we assume that the nonnegative eigenpair is positive (which is the case when $\mathcal{A}$ is irreducible). 
But the resulting algorithm will also be applicable in finding a nonnegative eigenpair with some zero components. 

Suppose that a nonnegative tensor $\mathcal{A}$ has a positive eigenpair $(\mathbf{x}_*, \lambda_*)$. 
We may try to find it by using Newton's method to solve $\mathbf{f}(\mathbf{x,}%
\lambda )=0$, where $\mathbf{f}$ is defined in \eqref{eq:Fx}. 
It is clear that  $\mathbf{J}\mathbf{f}(\mathbf{x}, \lambda)$,  
 the Jacobian of $\mathbf{f}$, satisfies a Lipschitz condition at $(\mathbf{x}_{\ast },
\lambda_{\ast} )$ since its Fr\'echet derivative is continuous in a
neighborhood of $(\mathbf{x}_{\ast}, \lambda_{\ast})$. 
We assume that  
\begin{equation}
\mathbf{Jf}(\mathbf{x}_*,\lambda_* )=\left[
\begin{array}{cc}
\lambda_* I-(m-1)T(\mathbf{x}_*) & \mathbf{x}_* \\
\mathbf{e}^{T} & 0%
\end{array}%
\right]   \label{eq: graFs}
\end{equation}%
is nonsingular. 
It is then well known that if $(\widehat{\mathbf{x}}_{0},\widehat{%
\lambda }_{0})$ is sufficiently close to $(\mathbf{x}_*, \lambda_*)$ then the Newton sequence 
$(\widehat{\mathbf{x}}_{k},\widehat{%
\lambda }_{k})$ is well defined and converges to $(\mathbf{x}_*, \lambda_*)$ quadratically. 
However,  if $(\widehat{\mathbf{x}}_{0},\widehat{%
\lambda }_{0})$ is not sufficiently close to $(\mathbf{x}_*, \lambda_*)$ the Newton sequence 
(if defined) usually does not converge to $(\mathbf{x}_*, \lambda_*)$ or any other positive eigenpair. 
We would like to present a modified Newton iteration that has guaranteed local quadratic convergence and has a good chance of finding 
a positive eigenpair starting from $(\widehat{\mathbf{x}}_{0},\widehat{%
\lambda }_{0})$, where $\widehat{\mathbf{x}}_{0}$ is any positive vector with unit $1$-norm and $\widehat{%
\lambda }_{0}$ is suitably chosen. 
To this end, we examine the Newton iteration more closely.

Given a positive pair $(\widehat{\mathbf{x}}_{k},\widehat{%
\lambda }_{k})$  sufficiently close to $(\mathbf{x}_*, \lambda_*)$, Newton's method produces the next approximation $(\widehat{%
\mathbf{x}}_{k+1},\widehat{\lambda }_{k+1})$ as follows:

\begin{align}
\left[
\begin{array}{cc}
\widehat{\lambda}_k I-(m-1)T({\widehat{\mathbf{x}}}_{k}) & \widehat{\mathbf{x%
}}_{k} \\
{\mathbf{e}}^{T} & 0%
\end{array}%
\right] \left[
\begin{array}{c}
\mathbf{d}_{k} \\
\delta _{k}%
\end{array}%
\right] & =\left[
\begin{array}{c}
\mathbf{r}(\widehat{\mathbf{x}}_{k},\widehat{\lambda }_{k}) \\
{\mathbf{e}}^{T}\widehat{\mathbf{x}}_{k}-1%
\end{array}%
\right],  \label{eq:step1} \\
\widehat{\mathbf{x}}_{k+1}& =\widehat{\mathbf{x}}_{k}\,-\mathbf{d}_{k},
\label{eq:step2} \\
\widehat{\lambda }_{k+1}& =\widehat{\lambda }_{k}-\delta _{k}.
\label{eq:step3}
\end{align}

We assume that $\widehat{\lambda}_k I-(m-1)T({\widehat{\mathbf{x}}}_{k})$ is nonsingular, but we do not assume that 
$\widehat{\lambda}_* I-(m-1)T({\widehat{\mathbf{x}}}_*)$ is nonsingular. 

Assuming ${\mathbf{e}}^{T}\widehat{%
\mathbf{x}}_{k}=1$, we use block Gaussian elimination in \eqref{eq:step1} to obtain
\begin{equation}
\left ({\mathbf{e}}^{T} \widehat{\mathbf{w}}_{k} \right ) \delta _{k}={\mathbf{e}}^{T} \left (\widehat{\lambda}_k I-(m-1)T({%
\widehat{\mathbf{x}}}_{k})\right )^{-1}\mathbf{r}(\widehat{\mathbf{x}}_{k},%
\widehat{\lambda }_{k}) ,
\label{eqdelta}
\end{equation}%
where we have let
\begin{equation}  \label{vw}
\widehat{\mathbf{w}}_{k}=\left (\widehat{\lambda}_k I-(m-1)T({\widehat{%
\mathbf{x}}}_{k})\right )^{-1}\widehat{\mathbf{x}}_{k}.
\end{equation}
Since
\begin{eqnarray}
\mathbf{r}(\widehat{\mathbf{x}}_{k},\widehat{\lambda }_{k}) &=&\frac{1}{m-1}%
\left ((m-2)\widehat{\lambda}_k\widehat{\mathbf{x}}_k+\widehat{\lambda}_k%
\widehat{\mathbf{x}}_k -(m-1)\mathcal{A}\widehat{\mathbf{x}}_k^{m-1}\right )
\nonumber \\
&=&\frac{1}{m-1}\left ((m-2)\widehat{\lambda}_k\widehat{\mathbf{x}}_k+\left (%
\widehat{\lambda}_k I -(m-1)T(\widehat{\mathbf{x}}_k) \right ) \widehat{%
\mathbf{x}}_k \right ), \label{rf1}
\end{eqnarray}
we have by \eqref{eqdelta}, \eqref{vw}, and ${\mathbf{e}}^{T}\widehat{%
\mathbf{x}}_{k}=1$ that
\begin{equation*}
\left ({\mathbf{e}}^{T} \widehat{\mathbf{w}}_{k} \right ) \left (\delta _{k}-\frac{m-2}{m-1}\widehat{\lambda }_{k}\right )=\frac{1}{m-1}. 
\end{equation*}%
Thus for $m\ge 3$ and  $(\widehat{\mathbf{x}}_{k},\widehat{%
\lambda }_{k})$  sufficiently close to $(\mathbf{x}_*, \lambda_*)$, 
${\mathbf{e}}^{T} \widehat{\mathbf{w}}_{k}\approx \frac{1}{(2-m)\lambda_*}$. In particular, 
${\mathbf{e}}^{T} \widehat{\mathbf{w}}_{k}<0$ and  
\begin{equation}
\delta _{k}=\frac{m-2}{m-1}\, \widehat{\lambda}_k +\frac{1}{(m-1){\mathbf{e}}%
^{T} \widehat{\mathbf{w}}_{k}}.  \label{eqdelta2}
\end{equation}%
Then by \eqref{eq:step1} and \eqref{vw}--\eqref{eqdelta2} we get 
\begin{equation}
\mathbf{d}_{k}=\frac{1}{m-1}\, \widehat{\mathbf{x}}_{k}-\frac{1}{(m-1) {%
\mathbf{e}}^{T}\widehat{\mathbf{w}}_{k}}\widehat{\mathbf{w}}_{k} .
\label{eqdk}
\end{equation}
Thus for $m\ge 3$ and  $(\widehat{\mathbf{x}}_{k},\widehat{%
\lambda }_{k})$  sufficiently close to $(\mathbf{x}_*, \lambda_*)$, 
$ \widehat{\mathbf{w}}_{k}\approx \frac{1}{(2-m)\lambda_*}\mathbf{x}_*$. In particular, 
$ \widehat{\mathbf{w}}_{k}<0$.
From  (\ref{eqdk}) and  (\ref{eqdelta2}), we
have
\begin{align}
\widehat{\mathbf{x}}_{k+1}& =\widehat{\mathbf{x}}_{k}\,-\mathbf{d}_{k}=
\frac{1}{m-1}\left ( (m-2) \widehat{\mathbf{x}}_{k}+\frac{1}{\mathbf{e}^{T}%
\widehat{\mathbf{w}}_{k}}\widehat{\mathbf{w}}_{k} \right ),  \label{eq:newtonup} \\
\widehat{\lambda }_{k+1}& =\widehat{\lambda }_{k}-\delta _{k}=\frac{1}{m-1}\left (
\widehat{\lambda }_{k}-\frac{1}{{\mathbf{e}}^{T}\widehat{\mathbf{w}}%
_{k} }\right ).  \label{ndown}
\end{align}

When $\widehat{\mathbf{w}}_{k}<0$, we have $\widehat{\mathbf{x}}_{k+1}>0$. 
However,  if $(\widehat{\mathbf{x}}_{k},\widehat{%
\lambda }_{k})$  is not sufficiently close to $(\mathbf{x}_*, \lambda_*)$, we do not always have  $\widehat{\mathbf{w}}_{k}<0$. 
In fact, it is possible to have the opposite: $\widehat{\mathbf{w}}_{k}>0$. In this case,  we also have $\widehat{\mathbf{x}}_{k+1}>0$. 

We now introduce some modifications to the Newton iteration. 

If  $\widehat{\mathbf{w}}_{k}$ has both positive and negative components, then we use a post-processing procedure, but avoid drastic changes. This is the intervention we mentioned in Section~\ref{sec1}. 
 Let $s_k=(\max \widehat{\mathbf{w}}_{k})( \min \widehat{\mathbf{w}}_{k})$. 
 We will use the following simple procedure:
\begin{equation}
{\mathbf{w}}_k=\left\{
\begin{array}{ll}
\max (\widehat{\mathbf{w}}_{k}, \mathbf{0})  & \text{if }  s_k<0 \text{ and }      |\max \widehat{\mathbf{w}}_{k}|>|\min \widehat{\mathbf{w}}_{k}|,                   \\
 \min (\widehat{\mathbf{w}}_{k}, \mathbf{0})             &  \text{if }  s_k<0 \text{ and }      |\max \widehat{\mathbf{w}}_{k}|\le |\min \widehat{\mathbf{w}}_{k}|,                   \\
\widehat{\mathbf{w}}_{k}   &  \text{if }  s_k\ge 0.
\end{array}%
\right.  \label{wupdate}
\end{equation}%
For example, $\widehat{\mathbf{w}}_{k}=[-100,1]^T$ will be updated to $[-100,0]^T$, rather than $[0,1]^T$. 
After $\widehat{\mathbf{w}}_{k}$ is updated to ${\mathbf{w}}_k$, we have $\widehat{\mathbf{x}}_{k+1}>0$ 
in \eqref{eq:newtonup}.

 Since the formula  \eqref{eq:newtonup} is derived under the assumption that  ${\mathbf{e}}^{T}\widehat{%
\mathbf{x}}_{k}=1$ and since we are looking for a positive $Z_1$-eigenvector,  $\widehat{\mathbf{x}}_{k+1}$ will immediately be normalized to  ${\mathbf{x}}_{k+1}>0$ with unit $1$-norm. For this reason, it is not necessary to keep the factor $1/(m-1)$ in  \eqref{eq:newtonup}.

Instead of using \eqref{ndown} to compute a new approximation to $\lambda_{\ast}$, we can take approximation 
$\lambda_{k+1}$ to be any value in the interval $[\underline{\lambda }_{k+1}, \overline{\lambda }_{k+1}],$
where 
\begin{equation}  \label{lamx}
\underline{\lambda }_{k+1} =\min \left( \frac{\mathcal{A}{\mathbf{x}}%
_{k+1}^{m-1}}{{\mathbf{x}}_{k+1}}\right), \quad 
\overline{\lambda }_{k+1} =\max \left( \frac{\mathcal{A}{\mathbf{x}}%
_{k+1}^{m-1}}{{\mathbf{x}}_{k+1}}\right), 
\end{equation}
such that ${\lambda}_{k+1} I-(m-1)T({\mathbf{x}_{k+1}})$ is not singular or nearly singular. 
The default value is $\lambda_{k+1}=\overline{\lambda }_{k+1}$, but a smaller value is to be used if 
$\overline{\lambda}_{k+1} I-(m-1)T({\mathbf{x}}_{k+1})$ is singular or nearly singular 
(We have not yet seen the need to do so in our experiments). 

We then have the following modified Newton iteration (Algorithm \ref%
{alg:Nini}) for finding a nonnegative eigenpair of a  nonnegative tensor $\mathcal{A}$. 

\begin{algorithm}
\begin{enumerate}
  \item   Given $\bx_0 > 0$ with $\Vert \bx_0\Vert_1 =1$, and ${\sf tol}>0$.
\item Compute  $\overline{\lambda}_{0}= \max \left( \frac{\mathcal{A}\mathbf{x}_{0}^{m-1}}{\mathbf{x}_{0}}\right)$ and 
 $\underline{\lambda}_{0}= \min \left( \frac{\mathcal{A}\mathbf{x}_{0}^{m-1}}{\mathbf{x}_{0}}\right)$. 
  \item   {\bf for} $k =0,1,2,\dots$    {\bf until} $\left\Vert \mathcal{A}\mathbf{x}_{k}^{m-1} - \overline{\lambda }_{k}\mathbf{x}_{k} \right\Vert_1  <{\sf tol}$.
\item \quad Choose $\lambda_{k}\in [\underline{\lambda }_{k}, \overline{\lambda }_{k}]$
such that ${\lambda}_{k} I-(m-1)T({\mathbf{x}_{k}})$ is nonsingular. 
  \item   \quad Solve the linear system $\left ( {\lambda}_k I- (m-1) T(\mathbf{x}_{k})\right) \widehat{\mathbf{w}}_{k}=\mathbf{x}_{k}$.
  \item   \quad Determine the vector $\bw_{k}$ by \eqref{wupdate}.
  \item   \quad Compute the vector $\widetilde{\mathbf{x}}_{k+1} =(m-2)\mathbf{x}_{k}\,+  \mathbf{w}_k/  (\mathbf{e}^T  \mathbf{w}_{k})$.
  \item   \quad Normalize the vector $\widetilde{\mathbf{x}}_{k+1}$:   $\bx_{k+1}= \widetilde{\mathbf{x}}_{k+1}/\Vert \widetilde{\mathbf{x}}_{k+1}\Vert_1$.
  \item   \quad Compute $\overline{\lambda }_{k+1} =\max \left( \frac{\mathcal{A}\mathbf{x}_{k+1}^{m-1}}{\mathbf{x}_{k+1}}\right)$ and
  $\underline{\lambda }_{k+1} =\min \left( \frac{\mathcal{A}\mathbf{x}_{k+1}^{m-1}}{\mathbf{x}_{k+1}}\right)$.
 
\end{enumerate}
\caption{Modified Newton iteration (MNI)}
\label{alg:Nini}
\end{algorithm}


Note that we have $\mathbf{x}_k>0$ during the iteration even when the algorithm is used to approximate a nonnegative eigenpair 
 $(\mathbf{x}_*, \lambda_*)$ with $\mathbf{x}_*$ having some zero components. Note also that we have $\underline{\lambda }_{k}<\overline{\lambda }_{k}$ in line 4 of the algorithm, so a suitable $\lambda_k$ can be chosen from the interval when 
$\lambda_k=\overline{\lambda}_k$ does not work (which should be a rare event).

\section{Local quadratic convergence of MNI}
\label{sec4}

In this section, we prove that the modified Newton iteration has local quadratic
convergence under the usual assumption that guarantees the  local quadratic
convergence of the original Newton iteration. 

 The following result is a direct consequence of a basic result of Newton's method; see \cite[Theorem 5.1.2]{K95} for example.

\begin{lemma}
\label{newton} Suppose that $\left( \mathbf{x}_{k}, {\lambda }%
_{k}\right) $ from Algorithm \ref{alg:Nini} is sufficiently close to a positive
eigenpair $\left( \mathbf{x}_{\ast }, \lambda_{\ast}\right)$ of a nonnegative tensor $\mathcal{A}$ and that 
the matrix in \eqref{eq: graFs} is nonsingular. Let $(\widehat{\mathbf{x}}_{k+1},
\widehat{\lambda}_{k+1})$ be obtained by Newton's method as in %
\eqref{eq:newtonup} and \eqref{ndown}, from $\left (\mathbf{x}_{k},
{\lambda }_{k}\right )$ instead of $\left (\widehat{\mathbf{x}}%
_{k}, \widehat{{\lambda }}_{k}\right )$. Then there is a constant $c_1$ such
that
\begin{equation}
\left\Vert \left[
\begin{array}{c}
\widehat{\mathbf{x}}_{k+1} \\
\widehat{\lambda }_{k+1}%
\end{array}%
\right] -\left[
\begin{array}{c}
\mathbf{x}_{\ast } \\
\lambda_{\ast}%
\end{array}%
\right] \right\Vert_1 \leq c_1 \left\Vert \left[
\begin{array}{c}
\mathbf{x}_{k} \\
{\lambda }_{k}%
\end{array}%
\right] -\left[
\begin{array}{c}
\mathbf{x}_{\ast } \\
\lambda_{\ast}%
\end{array}%
\right] \right\Vert_1^{2}.  \label{eq: qudraNT}
\end{equation}
\end{lemma}

\begin{remark}
We assume that $\mathbf{J}\mathbf{f}(\mathbf{x}_{\ast}, \lambda_{\ast})$ in  \eqref{eq: graFs} is
nonsingular, but we do not assume that ${\lambda}_{\ast} I- (m-1) T(\mathbf{x}_{\ast})$ is nonsingular. 
When $m=2$, the $Z$-eigenvalue problem here is the same as the $H$-eigenvalue problem studied in \cite{LGL17} for all $m\ge 2$, 
and it is shown there that  ${\lambda}_{\ast} I-  T(\mathbf{x}_{\ast})$ is always singular and $\mathbf{J}\mathbf{f}(\mathbf{x}_{\ast}, \lambda_{\ast})$ is always nonsingular. For $m\ge 3$, however, the difference of these two assumptions is not that big, but the assumption that ${\lambda}_{\ast} I- (m-1) T(\mathbf{x}_{\ast})$ is nonsingular is still the stronger assumption. 
Indeed, when ${\lambda}_{\ast} I- (m-1) T(\mathbf{x}_{\ast})$ is nonsingular (for $m\ge 3$), 
$\mathbf{J}\mathbf{f}(\mathbf{x}_{\ast}, \lambda_{\ast})$ in  \eqref{eq: graFs} is
nonsingular if and only if 
$-\mathbf{e}^T \left ({\lambda}_{\ast} I-(m-1)T({\mathbf{x}}_{\ast})\right
)^{-1}{\mathbf{x}}_{\ast} \ne 0$. 
Since  
$$\left ({\lambda}_{\ast} I-(m-1)T({\mathbf{x}}_{\ast})\right )\mathbf{x}_{\ast}
={\lambda}_{\ast}\mathbf{x}_{\ast}    -(m-1)\mathcal{A}\mathbf{x}_{\ast}^{m-1}=(2-m) {\lambda}_{\ast}\mathbf{x}_{\ast}, 
$$
we indeed have  
\begin{equation*}
-\mathbf{e}^T \left ({\lambda}_{\ast} I-(m-1)T({\mathbf{x}}_{\ast})\right
)^{-1}{\mathbf{x}}_{\ast} =-\mathbf{e}^T \frac{1}{(2-m)\lambda_{\ast}}{%
\mathbf{x}}_{\ast}=\frac{1}{(m-2)\lambda_{\ast}}\ne 0.
\end{equation*}
\end{remark}

We will also need the following simple relation between $\left |{\lambda }_{k}-\lambda_{\ast}\right |$ and  
$\|\mathbf{x}_{k}-\mathbf{x}_{\ast }\|_1$. 

\begin{lemma}
\label{relation} Let $\left( \mathbf{x}_{\ast }, \lambda_{\ast}\right)$ be
a positive  eigenpair of   a nonnegative tensor  $\mathcal{A}$. Let $\left\{ (\mathbf{x}_{k}, {\lambda }_{k})\right\}$
be generated by Algorithm \ref{alg:Nini}. Then there is a constant $c_2>0$
such that $\left |{\lambda }_{k}-\lambda_{\ast}\right | \le c_2 \|%
\mathbf{x}_{k}-\mathbf{x}_{\ast }\|_1$ for all $\mathbf{x}_{k}$ sufficiently
close to $\mathbf{x}_{\ast }$.
\end{lemma}

\begin{proof}
Since  $\lambda_{k}\in [\underline{\lambda }_{k}, \overline{\lambda }_{k}]$, we have
\begin{eqnarray*}
\left |{\lambda }_{k}-\lambda_{\ast}\right | &\le&\max \left | \frac{%
\mathcal{A} \mathbf{x}_{k}^{[m-1]}}{\mathbf{x}_{k}}-\frac{\mathcal{A}
\mathbf{x}_{\ast}^{[m-1]}}{\mathbf{x}_{\ast}}\right | \le \left \|
\frac{\mathcal{A} \mathbf{x}_{k}^{[m-1]}}{\mathbf{x}_{k}}-\frac{\mathcal{A}
\mathbf{x}_{\ast}^{[m-1]}}{\mathbf{x}_{\ast}} \right \|_1.
\end{eqnarray*}
Since the Fr\'echet derivative of $\frac{\mathcal{A}\mathbf{x}^{[m-1]}}{%
\mathbf{x}}$ is continuous in a neighborhood of $\mathbf{x}_{\ast}$, we have
$\left |{\lambda }_{k}-\lambda_{\ast}\right | \le c_2 \|\mathbf{x}%
_{k}-\mathbf{x}_{\ast }\|_1$ for a constant $c_2>0$.
\end{proof}

We now prove the local quadratic convergence of Algorithm \ref{alg:Nini}. 
We assume $m\ge 3$ since the result holds for $m=2$ by \cite{LGL17}. 

\begin{theorem}
\label{quadratic} 
 Let $\left( \mathbf{x}_{\ast }, \lambda_{\ast}\right)$ be
a positive  eigenpair of   a  nonnegative tensor  $\mathcal{A}$, with $\mathbf{J}\mathbf{f}(\mathbf{x}_{\ast}, \lambda_{\ast})$ in  \eqref{eq: graFs} being nonsingular, and let  $\left\{ (\mathbf{x}_{k}, {\lambda }_{k})\right\}$
be generated by Algorithm \ref{alg:Nini}. Suppose that $(%
\mathbf{x}_{k_0}, {\lambda}_{k_0})$ is sufficiently close to $\left( \mathbf{x}_{\ast }, \lambda_{\ast}\right)$ for some $k_0\ge 0$. 
Then $\mathbf{x}_{k}$ converges
to $\mathbf{x}_{\ast}$ quadratically and ${\lambda }_{k}$ converges
to $\lambda_{\ast}$ quadratically.
\end{theorem}

\begin{proof}
For some $\eta\in (0, \min \mathbf{x}_{\ast})$, there are positive constants $c_1$, $%
c_2 $ and $c_3$ such that
\begin{equation}  \label{eq2.1}
\left\Vert \left[
\begin{array}{c}
\widehat{\mathbf{x}}_{k+1} \\
\widehat{\lambda }_{k+1}%
\end{array}%
\right] -\left[
\begin{array}{c}
\mathbf{x}_{\ast } \\
\lambda_{\ast}%
\end{array}%
\right] \right\Vert_1 \leq c_1 \left\Vert \left[
\begin{array}{c}
\mathbf{x}_{k} \\
{\lambda }_{k}%
\end{array}%
\right] -\left[
\begin{array}{c}
\mathbf{x}_{\ast } \\
\lambda_{\ast}%
\end{array}%
\right] \right\Vert_1^{2}
\end{equation}
whenever $\left\Vert \left[
\begin{array}{c}
{\mathbf{x}}_{k} \\
{\lambda }_{k}%
\end{array}%
\right] -\left[
\begin{array}{c}
\mathbf{x}_{\ast } \\
\lambda_{\ast}%
\end{array}%
\right] \right\Vert_1 <\eta$ (by Lemma \ref{newton}),
\begin{equation}
\left |{\lambda }_{k}-\lambda_{\ast}\right | \le c_2 \|\mathbf{x}%
_{k}-\mathbf{x}_{\ast }\|_1  \label{eq2.2}
\end{equation}
whenever $\left\Vert {\mathbf{x}}_{k} -\mathbf{x}_{\ast } \right\Vert <\eta$ 
(by Lemma \ref{relation}), and
\begin{equation}
\left \| \left[
\begin{array}{c}
\mathbf{r}(\widehat{\mathbf{x}}_{k+1},\widehat{\lambda }_{k+1}) \\
{\mathbf{e}}^{T}\widehat{\mathbf{x}}_{k+1}-1%
\end{array}%
\right] - \left[
\begin{array}{c}
\mathbf{r}({\mathbf{x}}_{\ast},{\lambda }_{\ast}) \\
{\mathbf{e}}^{T}{\mathbf{x}}_{\ast}-1%
\end{array}%
\right] \right \|_1 \le c_3 \left\Vert \left[
\begin{array}{c}
\widehat{\mathbf{x}}_{k+1} \\
\widehat{\lambda }_{k+1}%
\end{array}%
\right] -\left[
\begin{array}{c}
\mathbf{x}_{\ast } \\
\lambda_{\ast}%
\end{array}%
\right] \right\Vert_1  \label{eq2.3}
\end{equation}
whenever $\left\Vert \left[
\begin{array}{c}
\widehat{\mathbf{x}}_{k+1} \\
\widehat{\lambda }_{k+1}%
\end{array}%
\right] -\left[
\begin{array}{c}
\mathbf{x}_{\ast } \\
\lambda_{\ast}%
\end{array}%
\right] \right\Vert_1 <\eta$ 
(since the Fr\'echet derivative of $\left[
\begin{array}{c}
\mathbf{r}({\mathbf{x}},{\lambda }) \\
{\mathbf{e}}^{T}{\mathbf{x}}-1%
\end{array}%
\right] $ is continuous). By the discussions leading to Algorithm \ref{alg:Nini}, we may also assume that $\widehat{\mathbf{w}}_k<0$ in line 5 of Algorithm \ref{alg:Nini} and thus ${\mathbf{w}}_k$ in line 7 of Algorithm \ref{alg:Nini} is still $\widehat{\mathbf{w}}_k$, whenever $\left\Vert \left[
\begin{array}{c}
{\mathbf{x}}_{k} \\
{\lambda }_{k}%
\end{array}%
\right] -\left[
\begin{array}{c}
\mathbf{x}_{\ast } \\
\lambda_{\ast}%
\end{array}%
\right] \right\Vert_1 <\eta$. 

 When $\left\Vert \left[
\begin{array}{c}
{\mathbf{x}}_{k} \\
{\lambda }_{k}%
\end{array}%
\right] -\left[
\begin{array}{c}
\mathbf{x}_{\ast } \\
\lambda_{\ast}%
\end{array}%
\right] \right\Vert_1 <\eta$, we have $\mathbf{x}_k>0$ and then $\widehat{\mathbf{x}}_{k+1}>0$ by \eqref{eq:newtonup}. 

Now we take
$$
\epsilon=\min \left (\eta,\  \sqrt{\frac{\eta}{c_1}}, \ \frac{1}{(1+c_3)c_1(1+c_2)^3} \right )
$$
and assume that $\left\Vert \left[
\begin{array}{c}
{\mathbf{x}}_{k} \\
{\lambda }_{k}%
\end{array}%
\right] -\left[
\begin{array}{c}
\mathbf{x}_{\ast } \\
\lambda_{\ast}%
\end{array}%
\right] \right\Vert_1 <\epsilon$ for $k=k_0$. 

By \eqref{eq2.2} we have 
\begin{equation}\label{eq99}
\left\Vert \left[
\begin{array}{c}
\mathbf{x}_{k} \\
{\lambda }_{k}%
\end{array}%
\right] -\left[
\begin{array}{c}
\mathbf{x}_{\ast } \\
\lambda_{\ast}%
\end{array}%
\right] \right\Vert_1=\|\mathbf{x}_k-\mathbf{x}_{\ast}\|_1+
|\lambda_k-\lambda_{\ast}|\le (1+c_2)\|{%
\mathbf{x}}_{k} - \mathbf{x}_{\ast } \|_1. 
\end{equation}
Then by \eqref{eq2.1}
\begin{equation*}
\|\widehat{\mathbf{x}}_{k+1} - \mathbf{x}_{\ast } \|_1 
\le c_1(1+c_2)^2\|{%
\mathbf{x}}_{k} - \mathbf{x}_{\ast } \|_1^2 , 
\end{equation*}
and also
$$
\left\Vert \left[
\begin{array}{c}
\widehat{\mathbf{x}}_{k+1} \\
\widehat{\lambda }_{k+1}%
\end{array}%
\right] -\left[
\begin{array}{c}
\mathbf{x}_{\ast } \\
\lambda_{\ast}%
\end{array}%
\right] \right\Vert_1 < c_1 \epsilon^2\leq \eta. 
$$
Then by \eqref{eq2.3}, \eqref{eq2.1} and \eqref{eq99}
\begin{equation*}
\left |\|\widehat{\mathbf{x}}_{k+1}\|_1 -1 \right | = \left |{\mathbf{e}}^{T}%
\widehat{\mathbf{x}}_{k+1} -1 \right | \le c_3c_1(1+c_2)^2\|{\mathbf{x}}_{k} -
\mathbf{x}_{\ast } \|_1^2 . 
\end{equation*}
Note that 
$$
{\mathbf{x}}_{k+1} =\frac{ \widetilde{\mathbf{x}}_{k+1}}{\|\widetilde{%
\mathbf{x}}_{k+1} \|_1}= \frac{ \widehat{\mathbf{x}}_{k+1}}{\|\widehat{%
\mathbf{x}}_{k+1}  \|_1}. 
$$
Then 
\begin{eqnarray*}
\|{\mathbf{x}}_{k+1} - \mathbf{x}_{\ast } \|_1&=&\| {\mathbf{x}}_{k+1} - 
\widehat{\mathbf{x}}_{k+1} + \widehat{\mathbf{x}}_{k+1} - \mathbf{x}_{\ast }
\|_1\\
&\le &\| {\mathbf{x}}_{k+1} - \widehat{\mathbf{x}}_{k+1}\|_1+\| \widehat{%
\mathbf{x}}_{k+1} - \mathbf{x}_{\ast } \|_1 \\
&=& \| \left ({\mathbf{x}}_{k+1} - \|\widehat{\mathbf{x}}_{k+1}\|_1{\mathbf{x%
}}_{k+1}\right ) \|_1+\| \widehat{\mathbf{x}}_{k+1} - \mathbf{x}_{\ast } \|_1
\\
&=&\left | \|\widehat{\mathbf{x}}_{k+1}\|_1-1\right | +\| \widehat{\mathbf{x}%
}_{k+1} - \mathbf{x}_{\ast } \|_1 \\
&\le & \left (1+c_3\right ) c_1(1+c_2)^2\|{\mathbf{x}}_{k} -
\mathbf{x}_{\ast } \|_1^2.
\end{eqnarray*}
By the choice of $\epsilon$ we have  $(1+c_3) c_1(1+c_2)^2 \epsilon \le
\frac{1}{1+c_2}$ and thus  $\|{\mathbf{x}}_{k+1} - \mathbf{x}_{\ast } \|_1 <(1+c_3) c_1(1+c_2)^2\epsilon^2\le 
\frac{1}{1+c_2}\epsilon < \eta$. Then  $\left |{{\lambda }}%
_{k+1}-\lambda_{\ast}\right | \le c_2 \|\mathbf{x}_{k+1}-\mathbf{x}_{\ast
}\|_1 < \frac{c_2}{1+c_2}\epsilon$. Therefore, $\left\Vert \left[
\begin{array}{c}
{\mathbf{x}}_{k+1} \\
{\lambda }_{k+1}%
\end{array}%
\right] -\left[
\begin{array}{c}
\mathbf{x}_{\ast } \\
\lambda_{\ast}%
\end{array}%
\right] \right\Vert_1 = \|{\mathbf{x}}_{k+1} - \mathbf{x}_{\ast } \|_1+
\left |{\lambda }_{k+1}-\lambda_{\ast}\right | < \epsilon$. We can then
repeat the above process to get $\|{\mathbf{x}}_{k+1} - \mathbf{x}_{\ast }
\|_1\le d |{\mathbf{x}}_{k} - \mathbf{x}_{\ast } \|_1^2$ for all $k\ge k_0$
and $d = \left (1+c_3\right ) c_1(1+c_2)^2$. Thus $\mathbf{x}_{k}$
converges to $\mathbf{x}_{\ast}$ quadratically and then ${\lambda }%
_{k}$ converges to $\lambda_{\ast}$ quadratically by \eqref{eq2.2}.
\end{proof}

\section{Application to transition probability tensors} 
\label{sec5}

In Algorithm \ref{alg:Nini}, we need to solve nonsingular 
linear systems of the form 
\begin{equation}
\left( \sigma I-(m-1)T(\mathbf{x})\right) \mathbf{w=x.}
\label{eq: linearsys3}
\end{equation}%
We assume $m\ge 3$. 
Suppose that $\left( \mathbf{x}, {\lambda }\right) $  is sufficiently close to a positive
eigenpair $\left( \mathbf{x}_{\ast }, \lambda_{\ast}\right)$ of  $\mathcal{A}$ and that 
the matrix in \eqref{eq: graFs} is nonsingular. Then we already know that $\mathbf{w}<0$ for the linear system, from 
the discussions leading to Algorithm \ref{alg:Nini}. 

In this section we will explain that, for transition probability tensors, 
 it is likely (but not guaranteed) that we always have $\widehat{\mathbf{w}}_k>0$ or $\widehat{\mathbf{w}}_k<0$ 
during the iteration, starting with $\mathbf{x}_0$ not necessarily close to $\mathbf{x}_{\ast}$. 

We start with the following result.

\begin{lemma}\label{negative}
Let $B$ be an $n\times n$  irreducible nonnegative matrix. If $\sigma<\rho(B)$ is sufficiently close to $\rho(B)$, 
then $(\sigma I-B)^{-1}<0$. 
\end{lemma}

\begin{proof}
By Perron--Frobenius theorem \cite{BPl94,V00}, $\rho(B)$ is a simple eigenvalue of $B$ with a positive unit eigenvector $\mathbf{u}$. 
Let 
$$
P^{-1}BP=\left[
\begin{array}{cc}
\rho (B) & 0 \\
0 & J
\end{array}\right]
$$
be the Jordan canonical form of $B$, where $P=\left [
\begin{array}{cc}
\mathbf{u} & U%
\end{array}%
\right]$ and $J$ is the direct sum of the Jordan blocks corresponding to eigenvalues other than $\rho(B)$. 

Let $P^{-1}=\left[
\begin{array}{cc}
\mathbf{v} & V%
\end{array}%
\right] ^{T}$. Then 
\begin{equation*}
\left[
\begin{array}{c}
\mathbf{v}^{T} \\
V^{T}%
\end{array}%
\right] B =\left[
\begin{array}{cc}
\rho (B) & 0 \\
0 & J
\end{array}\right]
\left[
\begin{array}{c}
\mathbf{v}^{T} \\
V^{T}%
\end{array}%
\right].
\end{equation*}%
Thus $\mathbf{v}$ is a left eigenvector of $B$ corresponding to $\rho(B)$. We have $\mathbf{v}>0$ since $\mathbf{v}^T\mathbf{u}=1$ by 
$P^{-1}P=I$.
Now 
$$
\sigma I-B=P\left[
\begin{array}{cc}
\sigma-\rho (B) & 0 \\
0 & \sigma I-J
\end{array}\right]P^{-1}, 
$$
and, when $\sigma$ is not an eigenvalue of $B$, 
\begin{eqnarray*}
(\sigma I-B)^{-1}&=&P\left[
\begin{array}{cc}
\sigma-\rho (B) & 0 \\
0 & \sigma I-J
\end{array}\right]^{-1}P^{-1}\\
&=&\left [
\begin{array}{cc}
\mathbf{u} & U%
\end{array}%
\right]  \left[
\begin{array}{cc}
(\sigma-\rho (B))^{-1} & 0 \\
0 & (\sigma I-J)^{-1}
\end{array}\right]\left[
\begin{array}{c}
\mathbf{v}^{T} \\
V^{T}%
\end{array}%
\right] \\
&=&(\sigma-\rho (B))^{-1}\mathbf{u}\mathbf{v}^{T} +U   (\sigma I-J)^{-1}   {V}^{T}. 
\end{eqnarray*}
It follows that  $(\sigma I-B)^{-1}<0$
when $\sigma<\rho(B)$ is sufficiently close to $\rho(B)$. 
\end{proof}

We now examine the sign pattern of the solution $\mathbf{w}$ of the linear system \eqref{eq: linearsys3}.

\begin{proposition}
\label{prop1} Let $\mathcal{A}$ be a transition probability tensor. Given a vector $\mathbf{x}>0$ with $\left\Vert \mathbf{x}\right\Vert _{1}=1$ and
consider the linear system \eqref{eq: linearsys3}.
Then

\begin{enumerate}
\item If $\sigma >m-1$ then $\mathbf{w}>0.$
\item If $\sigma <m-1$ then  $\mathbf{e}^{T}\mathbf{w}<0$ 
(so $\mathbf{w}$ has at least one negative components). 
\item If $\sigma <m-1$ is sufficiently close to $m-1$ and  $\mathcal{A}$ is weakly  irreducible, then  $\mathbf{w}<0$. 
\end{enumerate}
\end{proposition}

\begin{proof}
We have $\rho(T(\mathbf{x}))=1$ by Lemma \ref{leftev}. 
If $\sigma >m-1$, then $\sigma I-(m-1)T(\mathbf{x})$ 
 is a nonsingular $M$-matrix, and thus $\mathbf{w}=\left( \sigma I-(m-1)T(\mathbf{x})\right)^{-1}\mathbf{x} >0.$ 

By Lemma \ref{leftev} we also have 
\begin{equation*}
\mathbf{e}^{T}\mathbf{x}
=\mathbf{e%
}^{T}\left( \sigma I-(m-1)T(\mathbf{x})\right)\mathbf{w}=\left( \sigma
-(m-1)\right)\mathbf{e}^{T}\mathbf{w}.  
\end{equation*}%
If $\sigma<m-1$, then $\mathbf{e}^{T}\mathbf{w}<0$.

When ${\mathcal A}$ is weakly  irreducible, $T(\mathbf{x})$ is an irreducible nonnegative matrix by Lemma \ref{irreducible}. If  $\sigma <m-1$ is sufficiently close to $m-1$, then $\left( \sigma I-(m-1)T(\mathbf{x})\right)^{-1}<0$ by Lemma \ref{negative}
and thus $\mathbf{w}=\left( \sigma I-(m-1)T(\mathbf{x})\right)^{-1}\mathbf{x} <0$. 
\end{proof}

When Algorithm \ref{alg:Nini} is applied to a  transition probability tensor, we have $\lambda_k\ne m-1$ in line 4 of the algorithm. 
The algorithm typically requires a small number of iterations for convergence to a positive eigenpair $\left( \mathbf{x}_{\ast }, \lambda_{\ast}\right)$. Note that we always have $\lambda_{\ast}=1$ for a  transition probability tensor.
Suppose that $\left( \mathbf{x_k}, {\lambda_k}\right) $ in Algorithm \ref{alg:Nini}  is sufficiently close to $\left( \mathbf{x}_{\ast }, \lambda_{\ast}\right)$  and that 
the matrix in \eqref{eq: graFs} is nonsingular. Then we already know that $\widehat{\mathbf{w}}_k<0$ in line 5 of the  algorithm.
Now Proposition \ref{prop1} tells us that  $\widehat{\mathbf{w}}_k>0$ if $\lambda_k>m-1$ and that $\widehat{\mathbf{w}}_k<0$ if $\lambda_k<m-1$ is close to $m-1$. Recall that $\overline{\lambda}_k\ge 1$ by Lemma \ref{minmax} and that we take 
$\lambda_k$ to be equal to $\overline{\lambda}_k$ or (to avoid singularity) to be sightly smaller than $\overline{\lambda}_k$. 
When $m=3$ for example, that $\lambda_k<m-1$ indicates that $\mathbf{x}_k$ is already not too far away from a positive eigenvector. We then have a good chance of having $\widehat{\mathbf{w}}_k<0$ when $\lambda_k<m-1$. 
In this case, the intervention in the first two cases of \eqref{wupdate} is applied only occasionally. 

\section{Numerical experiments}
\label{sec6}

In this section we present some numerical results to show the usefulness of MNI. 
To accommodate the computation of eigenvectors with some zero components, we modify the computation of 
$\overline{\lambda}_k$ and $\underline{\lambda}_k$ as follows: 
$$
\overline{\lambda}_k = \max \left( \frac{(\mathcal{A}\mathbf{x}_{k}^{m-1})_{i}}{(\mathbf{x}_{k})_{i}} \ | \ (\mathbf{x}_{k})_{i}\geq 10^{-13} \right), \quad 
\underline{\lambda}_k = \min \left( \frac{(\mathcal{A}\mathbf{x}_{k}^{m-1})_{i}}{(\mathbf{x}_{k})_{i}}\ | \ (\mathbf{x}_{k})_{i}\geq 10^{-13} \right). 
$$
This will not cause any problem to MNI since in MNI we choose $\lambda_k$ from the true interval $[\underline{\lambda}_k, \overline{\lambda}_k]$, which contains the computed interval $[\underline{\lambda}_k, \overline{\lambda}_k]$ (with the above 
modification). 
So we can choose any $\lambda_k$ from the computed interval $[\underline{\lambda}_k, \overline{\lambda}_k]$ to ensure that 
the linear system in step 4 of MNI is not (nearly) singular. In our experiments, $\lambda_k=\overline{\lambda}_k$ always works. 
The default initial vector for MNI is $\mathbf{x}_{0}=\mathbf{e}/n$. But to find different nonnegative eigenpairs, we 
run MNI  a number of times using $%
\mathbf{x}_{0}=\mathbf{y}_0/\|\mathbf{y}_0\|_1$ with $\mathbf{y}_0={\rm rand(n,1)}$ in MATLAB.
We terminate the iteration when the residual is small enough:  $\left\Vert \mathcal{A}\mathbf{x}_{k}^{m-1} - \overline{\lambda }_{k}\mathbf{x}_{k} \right\Vert < 10^{-13}$.

\begin{example}\label{ex1}
(Example 2.7 of \cite{CPZ13}) 
Let $\mathcal{A}\in \mathbb{R}_{+}^{[4,2]}$ be
defined by
\begin{align*}
A_{1111}&=A_{2222}=\frac{4}{\sqrt{3}},\ \ A%
_{1112}=A_{1121}=A_{1211}=A_{2111}=1, \\
A_{1222}&=A_{2122}=A_{2212}=A%
_{2221}=1, \text{ and }A_{ijkl}=0 \text{ elsewhere}.
\end{align*}
\end{example}
The tensor is irreducible and has three positive $Z$-eigenpairs:
\begin{align*}
 (\mathbf{x}^{(1)}, \lambda^{(1)})&=\left (\left [ \frac{1}{2}, \frac{1}{2}\right ]^T,   1+\frac{1}{\sqrt{3}}  \right ),   \\
(\mathbf{x}^{(2)}, \lambda^{(2)})&=\left (\left [ \frac{\sqrt{3}}{1+\sqrt{3}}, \frac{1}{1+\sqrt{3}}\right ]^T,   \frac{11}{3+2 \sqrt{3}}  \right ),    \\
(\mathbf{x}^{(3)}, \lambda^{(3)})&=\left (\left [ \frac{1}{1+\sqrt{3}}, \frac{\sqrt{3}}{1+\sqrt{3}}\right ]^T,   \frac{11}{3+2 \sqrt{3}}  \right ).
\end{align*}
Note that we have converted the $Z_2$-eigenpairs reported in \cite{CPZ13} to $Z_1$-eigenpairs here. 

For this example, we generate $5000$  random vectors $\mathbf{y}_0$, normalize them to $\mathbf{x}_0$, and apply MNI. 
Each time, the sequence $(\mathbf{x}_k, \lambda_k)$ from the algorithm converges to one of the three eigenpairs. 
In Table \ref{table1}, \textquotedblleft Occurrence\textquotedblright denotes the number of occurrences with convergence to a particular eigenpair. For each eigenpair, \textquotedblleft A-Sign\textquotedblright denotes the average number of times with $s_k=(\max \widehat{\mathbf{w}}_{k})( \min \widehat{\mathbf{w}}_{k})<0$ (This tells us how often the intervention in the first two cases of \eqref{wupdate} is needed), 
\textquotedblleft A-Iter\textquotedblright denotes the average number of iterations to
achieve convergence, \textquotedblleft A-Err\textquotedblright denotes the
average residual error when the iteration is terminated.

From Table \ref{table1}, we can see that, for a random initial vector $\mathbf{x}_0$, MNI would compute one of the positive eigenpairs 
quickly and accurately, with minimal intervention from \eqref{wupdate}. 

\begin{table}[tbp]
\caption{Numerical results for Example \ref{ex1}}
\label{table1}
\begin{center}
\begin{tabular}{l|rrrl}
\hline
$(\mathbf{x}, \lambda)$ & Occurrence & A-Sign & A-Iter & A-Err  \\ \hline
$  (\mathbf{x}^{(1)}, \lambda^{(1)})  $ & 967 & 0 & 4.12 & 5.73e-15  \\
$ (\mathbf{x}^{(2)}, \lambda^{(2)}) $ & 1966 & 0.10 & 7.10  & 6.58e-15 \\
$ (\mathbf{x}^{(3)}, \lambda^{(3)})$  & 2067 & 0.11 & 7.12  & 7.73e-15 \\ \hline
\end{tabular}%
\end{center}
\end{table}

\begin{example}
\label{ex2}
(Example 5.1 of \cite{CPZ13}) 
Let $\mathcal{A}\in \mathbb{R}_{+}^{[4,2]}$ be
defined by
\begin{align*}
A_{1111}&=1.1, \ \ A_{2222}=1.2,\ \ A%
_{1112}=A_{1121}=A_{1211}=A_{2111}=0.25, \\
A_{1222}&=A_{2122}=A_{2212}=A%
_{2221}=0.25, \text{ and }A_{ijkl}=0 \text{ elsewhere}.
\end{align*}
\end{example}
The tensor is irreducible and has three positive $Z$-eigenpairs:
\begin{align*}
 (\mathbf{x}^{(1)}, \lambda^{(1)})&\approx \left (\left [  0.1785, 0.8215        \right ]^T,   0.9216      \right ),   \\
(\mathbf{x}^{(2)}, \lambda^{(2)})&\approx \left (\left [  0.8052, 0.1948 \right ]^T,   0.8331  \right ),    \\
(\mathbf{x}^{(3)}, \lambda^{(3)})&\approx \left (\left [ 0.5193, 0.4807 \right ]^T,   0.5373  \right ).
\end{align*}
Note that we have converted the $Z_2$-eigenpairs reported in \cite{CPZ13} to $Z_1$-eigenpairs here. 

For this example, we again use $5000$  random initial vectors. The numerical results in Table \ref{table2} are similar to those in Table \ref{table1}, but we need the intervention in \eqref{wupdate} more often this time. 

\begin{table}[tbp]
\caption{Numerical results for Example \ref{ex2}}
\label{table2}
\begin{center}
\begin{tabular}{l|rrrl}
\hline
$(\mathbf{x}, \lambda)$ & Occurrence & A-Sign & A-Iter & A-Err  \\ \hline
$  (\mathbf{x}^{(1)}, \lambda^{(1)})  $ & 1288 & 0.37 & 5.79 & 4.20e-15  \\
$ (\mathbf{x}^{(2)}, \lambda^{(2)}) $ & 1245 & 0.35 & 5.77  & 4.96e-15 \\
$ (\mathbf{x}^{(3)}, \lambda^{(3)})$  & 2467 & 0.13 & 4.76  & 5.01e-15 \\ \hline
\end{tabular}%
\end{center}
\end{table}

\begin{example}
\label{ex3}
Consider  $\mathcal{A}\in \mathbb{R}_{+}^{[4,2]}$ defined by
$$
A_{1111}= 1.1,\ \   A_{2222}=1.2,\ \  A_{1112}=A_{1222}=0.25, \text{ and }A_{ijkl}=0 \text{ elsewhere}.
$$
\end{example}
The tensor is not weakly irreducible and has three nonnegative $Z$-eigenpairs, two of them are positive:
\begin{align*}
 (\mathbf{x}^{(1)}, \lambda^{(1)})&\approx \left (\left [  0.1874, 0.8126     \right ]^T,   0.7923      \right ),   \\
(\mathbf{x}^{(2)}, \lambda^{(2)})&= \left (\left [  1, 0 \right ]^T,   1.1  \right ),    \\
(\mathbf{x}^{(3)}, \lambda^{(3)})&\approx \left (\left [0.4412, 0.5588     \right ]^T,  0.3746  \right ).
\end{align*}

For this example, we use $5000$  random initial vectors. From the numerical results in Table \ref{table3}, we can see that 
MNI takes many more iterations to approximate the second eigenpair, which has a zero component in the eigenvector. 
This is not too surprising since the local quadratic convergence of MNI is proved only for approximating positive eigenpairs. 

\begin{table}[tbp]
\caption{Numerical results for Example \ref{ex3}}
\label{table3}
\begin{center}
\begin{tabular}{l|rrrl}
\hline
$(\mathbf{x}, \lambda)$ & Occurrence & A-Sign & A-Iter & A-Err  \\ \hline
$  (\mathbf{x}^{(1)}, \lambda^{(1)})  $ & 1166 & 0.29 & 5.82 & 4.62e-15  \\
$ (\mathbf{x}^{(2)}, \lambda^{(2)}) $ & 876 & 69.17 & 69.17  & 8.27e-14 \\
$ (\mathbf{x}^{(3)}, \lambda^{(3)})$  & 2958 & 0.08 & 4.40  & 4.26e-15 \\ \hline
\end{tabular}%
\end{center}
\end{table}

We now perform some experiments on some transition probability tensors.
\begin{example}(Example 1.7 of \cite{CZ13}) 
\label{ex4}
Consider the transition probability tensor $\mathcal{P}\in \mathbb{R}_{+}^{[4,2]}$  given by
\begin{equation*}
\begin{array}{cccc}
P_{1111}=0.872 & P_{1112}=2.416/3 & P_{1121}=2.416/3 & P_{1122}=0.616/3 \\
P_{1211}=2.416/3 & P_{1212}=0.616/3 & P_{1221}=0.616/3 & P_{1222}=0.072 \\
P_{2111}=0.128 & P_{2112}=0.584/3 & P_{2121}=0.584/3 & P_{2122}=2.384/3 \\
P_{2211}=0.584/3 & P_{2212}=2.384/3 & P_{2221}=2.384/3 & P_{2222}=0.928 \  .
\end{array}%
\end{equation*}
\end{example}
The tensor has two positive $Z$-eigenpairs:
$$
(\mathbf{x}^{(1)}, \lambda^{(1)})= \left (\left [  0.6, 0.4     \right ]^T,   1      \right ),   \quad
(\mathbf{x}^{(2)}, \lambda^{(2)})= \left (\left [ 0.2, 0.8 \right ]^T,   1  \right ).
$$

For this example, we use $5000$  random initial vectors. From the numerical results in Table \ref{table4}, we can see that 
MNI takes  more iterations to approximate the first eigenpair. We then take two different random initial vectors, with MNI convergence to the two eigenpairs, and plot in Figure \ref{fig:MNIEx4} the eigenvector errors
$\|\mathbf{x}_k^{(i)}-\mathbf{x}^{(i)}\|_1, \ i=1, 2$. We see that 
the convergence of MNI is linear for the first eigenvector and is quadratic for the second eigenvector. 
The reason is that the matrix in \eqref{eq: graFs} is singular at the first eigenpair and is nonsingular at the second eigenpair. 
We then compare MNI with the SS-HOPM algorithm with $\alpha=1$ \cite{KM11} in Table \ref{table4}, with the same initial vector and same stopping criterion for each trial. We find that SS-HOPM fails to satisfy the stopping criterion within 10000 iterations 
for approximating the first eigenpair.

\begin{table}[tbp]
\caption{Numerical results for Example \ref{ex4}}
\label{table4}
\begin{center}
\begin{tabular}{l|rrrl}
\hline
 MNI   & Occurrence   & A-Sign & A-Iter & A-Err\\ \hline
$(\mathbf{x}^{(1)}, \lambda^{(1)}) $  & 3620 & 0.13 & 18.54 & 5.14e-14 \\
$(\mathbf{x}^{(2)}, \lambda^{(2)})  $ & 1380 & 0.34 & 5.63  &6.17e-15 \\ \hline \hline
SS-HOPM with $\alpha = 1$&   &   &   &   \\ \hline
$(\mathbf{x}^{(1)}, \lambda^{(1)}) $  & 1764 &  & 10000 & 5.31e-07  \\
$(\mathbf{x}^{(2)}, \lambda^{(2)}) $  & 3236 &  & 392.4  &9.53e-14 \\ \hline
\end{tabular}%
\end{center}
\end{table}

\begin{figure}[hbtp]
\centering
\epsfig{file=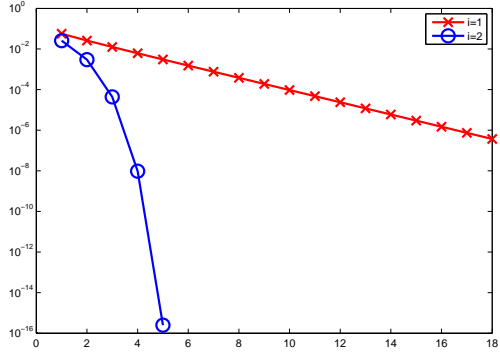,width=.6\textwidth}
\caption{Convergence of $\|\mathbf{x}_k^{(i)}-\mathbf{x}^{(i)}\|_1$  for Example \ref{ex4}.}
\label{fig:MNIEx4}
\end{figure}

\begin{example}(Example 1.5 of \cite{CZ13}) 
\label{ex5}
Consider the transition probability tensor $\mathcal{P}\in \mathbb{R}_{+}^{[3,3]}$  given by
\begin{align*}
P_{111}&=P_{222}=P_{333}=1,\\
P_{122}&=P_{133}=P_{211}=P_{233}=P_{311}=P_{322}=0,  \text{ and }
P_{ijk}=1/3 \text{ elsewhere}.
\end{align*}
\end{example}
The tensor is weakly irreducible and has four nonnegative $Z$-eigenpairs, only one of them is positive:
\begin{align*}
(\mathbf{x}^{(1)}, \lambda^{(1)})&= \left (\left [ 1/3, 1/3, 1/3    \right ]^T,   1      \right ),   \\
(\mathbf{x}^{(2)}, \lambda^{(2)})&= \left (\left [ 1, 0, 0 \right ]^T,   1  \right ), \\
(\mathbf{x}^{(3)}, \lambda^{(3)})&= \left (\left [  0, 1, 0   \right ]^T,   1      \right ),   \\
(\mathbf{x}^{(4)}, \lambda^{(4)})&= \left (\left [ 0, 0, 1\right ]^T,   1  \right ). 
\end{align*}

For this example, we use $5000$  random initial vectors.  Table \ref{table5} shows that MNI computes the positive eigenpair 
$(\mathbf{x}^{(1)}, \lambda^{(1)})$ every time. 
\begin{table}[tbp]
\caption{Numerical results for Example \ref{ex5}}
\label{table5}
\begin{center}
\begin{tabular}{l|rrrl}
\hline
$(\mathbf{x}, \lambda)$ & Occurrence & A-Sign & A-Iter & A-Er \\ \hline
$(\mathbf{x}^{(1)}, \lambda^{(1)})$ & 5000 & 0.0012 & 6.54 & 5.14e-15  \\ \hline
\end{tabular}%
\end{center}
\end{table}

Finally, we consider the application of MNI to a transition probability tensor arising from the study of the multilinear PageRank problem \cite{GLY15}. 
\begin{example}
\label{R63} 
Let ${\bf {\it R}}_{6,3}$ be the matrix given in \cite[p. 1539]{GLY15}. Normalize each column of  ${\bf {\it R}}_{6,3}$ to 
get a column stochastic matrix $[S_1 \ S_2  \ S_3 \ S_4 \ S_5 \ S_6]$, where $S_k\in \mathbb R^{6\times 6}$ for each $k$. 
We consider the transition probability tensor $\mathcal{A}(\alpha)\in \mathbb{R}_{+}^{[3,6]}$  whose entries are given by 
$$
(\mathcal{A}(\alpha))_{ijk}=\alpha (S_k)_{ij}+(1-\alpha)(\mathbf{v})_i, 
$$
where we use $\mathbf{v}=\mathbf{e}/6$. 
\end{example}

We are going to find a nonnegative eigenvector of  $\mathcal{A}(\alpha)$ corresponding to eigenvalue $1$. 
All algorithms tested in \cite{GLY15}, with the default settings,  run into difficulties on this example when $\alpha=0.99$, and it is remarked in \cite{GLY15} that 
this test problem should be a useful case for future algorithmic studies on the multilinear PageRank problem.

The tensor is positive (and thus irreducible) for $0<\alpha<1$ and has a unique positive eigenvector for each $\alpha$ value in Table 
\ref{table6}, other than $\alpha=1$. 
When $\alpha=1$, the tensor is weakly irreducible and has a unique nonnegative eigenvector 
$[0,0,0,1,0,0]^T$. 
\begin{table}[tbph]
\caption{Numerical results for Example~\protect\ref{R63}}
\label{table6}\centering
\begin{tabular}{rlrrl}
\hline
Tensor $\mathcal{A}(\alpha)$   & & & MNI &  \\ 
\cline{1-1}\cline{3-5}
  $\alpha $ & & Sign & Iter & Err \\ \hline 
  $0.1$ &  & $0$ & $4$ & 3.17e-15 \\ 
  $0.3$ &  & $0$ &$5$ & 1.20e-16 \\ 
  $0.5$ &  & $0$ &$6$ & 1.41e-16\\  
  $0.7$ &  & $0$ &$6$ & 1.43e-16 \\  
  $0.9$ &  & $0$ &$7$ & 1.73e-16\\  
  $0.99$ &  & $7$ &$20$ & 4.44e-16\\ 
  $0.999$ &  & $15$ &$30$ & 5.55e-15  \\ 
  $0.9999$ &  & $15$ &$23$ & 4.09e-14 \\ 
  $0.99999$ &  & $25$ &$46$ & 3.45e-14 \\
  $1$ &  & $87$ &$93$ & 6.31e-14 \\ \hline
\end{tabular}%
\end{table}
\begin{table}[pbt]
\caption{Numerical results for Example \ref{R63}}
\label{table1.1}
\begin{center}
\begin{tabular}{l|rr}
\hline
 & $\alpha = 0.99$ & $\alpha = 1$ \\ \hline
Eigenvector & 0.043820721946272 & 0.000000000000076  \\
 & 0.002224192630620 & 0.000000000000000   \\
 & 0.009256490884022 & 0.000000000000000   \\ 
  & 0.819168263512464 & 0.999999999999696   \\
 & 0.031217440669761 & 0.000000000000076   \\ 
 & 0.094312890356862 & 0.000000000000152   \\ \hline
\end{tabular}%
\end{center}
\end{table}
 
For this example, we apply MNI with the initial vector $\mathbf{x}_{0}=\mathbf{e}/6$.
In Table \ref{table6}, \textquotedblleft Sign\textquotedblright denotes the number of times with $s_k=(\max \widehat{\mathbf{w}}_{k})( \min \widehat{\mathbf{w}}_{k})<0$, \textquotedblleft Iter\textquotedblright denotes the number of iterations to
achieve convergence, \textquotedblleft Err\textquotedblright denotes the
residual error when the iteration is terminated. 
As suggested by our analysis in Section \ref{sec5}, we have $\widehat{\mathbf{w}}_{k}>0$ or $ \widehat{\mathbf{w}}_{k}<0$ 
during the iteration for $\alpha\le 0.9$ in the table. The case $\alpha=0.99$ does not pose any serious challenge to MNI, with the default initial vector. 
The number of iterations for $\alpha=0.99$ is larger than that for $\alpha=0.9$ for example. This is because, as $\alpha\to 1^-$, 
some components of the positive eigenvector are close to $0$, and we have already seen in Example \ref{ex3} that MNI will require 
more iterations when computing a nonzero eigenvector with one or more zero components. We have displayed the eigenvectors 
computed by MNI for $\alpha=0.99$ and $\alpha=1$ in Table \ref{table1.1}. Notice that the eigenvector for $\alpha=0.99$ is exactly the same as 
reported in \cite[p. 1534]{GLY15}.

%

\section{Conclusion}
\label{sec7}
We have proposed a modified Newton iteration (MNI) for finding a nonnegative $Z$-eigenpair of a nonnegative tensor. 
We have proved local quadratic convergence of MNI to any positive eigenpair of a nonnegative tensor when the Jacobian (for the original Newton iteration)   is nonsingular  at the eigenpair.   
Numerical experiments show that MNI can also be used to compute a positive eigenpair at which the Jacobian is singular, or to compute a nonnegative eigenpair with some zero components in the eigenvector, although no convergence theory has been established in those situations.  
When the tensor has both positive eigenpairs and nonnegative eigenpairs with some zero components in the eigenvector, MNI seems to find a positive eigenpair more often. We have not yet found any examples for which MNI (with the default initial vector)
fails to find a nonnegative $Z$-eigenpair of a nonnegative tensor,  but MNI should be more useful when computing a positive eigenpair of an irreducible nonnegative tensor, particularly when the Jacobian at the eigenpair is nonsingular. 

\section*{Acknowledgment}

This work was started when C.-H. Guo visited ST Yau Center at Chiao-Da
in Taiwan in late 2015; he thanks the Center for
its hospitality.

\end{document}